\documentclass[a4paper,12pt]{article}

\usepackage{amsmath,amssymb,amsthm,fancybox,color, comment}

\newtheorem{thm}{Theorem}
\newtheorem{prop}{Propositon}
\newtheorem{cor}{Corollary}
\newtheorem{lemm}{Lemma}

\newtheorem{ass}{Assumption}

\theoremstyle{definition}
\newtheorem{rem}{Remark}
\makeatletter
\numberwithin{equation}{section}
\allowdisplaybreaks

\makeatother

\def\Delta{\delta}

\begin{document}
\title{Asymptotic formulas for the logarithm of general $L$-functions at 
and near $s=1$}
\author{Kohji Matsumoto and Yumiko Umegaki}
\date{}
\maketitle
\begin{abstract}
We consider a general form of $L$-function $L(s)$ defined by an Euler product and
satisfies several analytic assumptions.     We show several asymptotic formulas for
$L(1)$ and $\log L(1)$.     In particular those asymptotic formulas are valid for
Dirichlet $L$-functions attached to almost all Dirichlet characters.   Our theorems
should be compared with former results due to Elliott, Montgomery and Weinberger, 
etc.
\end{abstract}
\section{Introduction}
\par
The aim of the present paper is to obtain a certain asymptotic formula for the values of
general $L$-functions at and near the point $s=1$.   
The original motivation of the authors was to consider
the values of symmetric power $L$-functions \cite{MU2018}, but here, we begin with
recalling known facts in the case of classical Dirichlet $L$-functions.
\par
Let $\chi$ be a non-principal Dirichlet character mod $q$, and $L(s,\chi)$ the associated 
Dirichlet $L$-function.
It is not difficult to obtain (see \cite[Theorem~4.11 (c)]{MV}) that
\begin{align}\label{MV}
\log L(1,\chi)= \sum_{p\leq y}\frac{\chi (p)}{p} + C(1,\chi)
+O_{\chi}\bigg(\frac{1}{\log y}\bigg),
\end{align}
where (and in what follows) $p$ denotes prime numbers, and
\[
C(1,\chi)=\sum_{p:{\rm prime}}\sum_{k=2}^{\infty}\frac{\chi(p^k)}{kp^k}.
\]
However the above error term depends on $\chi$, hence on $q$.   It is an important problem 
how to obtain error estimates uniform in $q$.
In 1977, Montgomery and Weinberger proved that
\begin{prop}{\rm (Montgomery and Weinberger \cite[Lemma 2]{MW})}\label{lem:MW}
  Suppose that $0<\Delta<1$. Then for $(\log q)^{\Delta}\leq y \leq \log q$,
  and $\chi$ a primitive character $\pmod{q}$, $q>1$,
  \[
  \log L(1,\chi)=\sum_{p\leq y}\frac{\chi(p)}{p}+O_{\Delta}(1)
  \]
  unless $\chi$ lies in an exceptional set $E(\Delta)$.
  The set $E(\Delta)$ contains $\ll Q^{\Delta}$
  primitive characters $\chi$ with conductor $q\leq Q$.
\end{prop}
\begin{rem}\label{rem_MW}
The definition of $E(\delta)$ in \cite{MW} is that it consists of all primitive characters
$\chi$ mod $q$, $q\leq Q$, for which $L(s,\chi)$ has at least one zero in the rectangle
$1-\delta/7\leq \sigma\leq 1$, $|t|\leq (\log q)^2$.
  Note that the implied constant in the above estimation of $\# E(\Delta)$
  depends on $\Delta$.
  Actually, in the proof of this result in \cite{MW},
  they use the large sieve inequality
  (see \cite[Theorem 12.2]{Mont})
\begin{align}\label{zero-density}
  \sum_{q\leq Q}{\sum_{\chi}}^* N(\sigma, T,\chi)
  \ll (Q^2T)^{3(1-\sigma)}(\log QT)^9,
\end{align}
where $N(\sigma,T,\chi)$ denotes the number of zeros of $L(s,\chi)$ in the region
$\sigma\leq\Re s\leq 1$, $|\Im s|\leq T$, and the asterisk means that the summation
  runs over all primitive characters mod $q$.    They apply this inequality
  with $\sigma=1-\Delta/7$ and $T=(\log Q)^2$ 
to show that
\begin{align}\label{estimate_E}
  \# E(\Delta)
  \ll
  (Q\log Q)^{6\Delta/7}(\log Q)^9
  \ll_{\Delta} Q^{\Delta}.
\end{align}
In Morita et al. \cite{MUU}, a result similar to the above proposition but $y$ is replaced by
$h\log Q$ ($0 < h \leq 1$) is shown.
\end{rem}

\par
Another type of evaluation of $L(1,\chi)$ was obtained by
Elliott~\cite{Ecrelle} \cite{E}.
We state his result in the form given in Granville and Soundararajan 
\cite[Proposition 2.2]{GS2003}.
\begin{prop}[Elliott]\label{lem:E}
Let $Q$ be large, and $1\leq A\leq \log\log Q$.   Then for all but at most $Q^{2/A}$
primitive characters $\chi \pmod{q}$ with $q\leq Q$ we have
\[
L(1,\chi)=
\left(1+O\left((A^2+\log^2 y)\frac{\log Q}{y^{1/4A}}\right)\right)
\prod_{p\leq y}\left(1-\frac{\chi(p)}{p}\right)^{-1}
\]
for $y\leq Q/2$, and further
\[
L(1,\chi)=
\left(1+O\left(\frac{1}{\log\log Q}\right)\right)
\prod_{p\leq (\log Q)^A}\left(1-\frac{\chi(p)}{p}\right)^{-1}
\]
holds for all but at most $Q^{2/A+5\log\log\log Q/\log\log Q}$ primitive characters
$\chi$ (mod $q$) with $q\leq Q$.
\end{prop}
\par
In 2001, on the values of Dirichlet $L$-functions in a certain rectangle, 
Granville and Soundararajan obtained:
\begin{prop}{\rm (Granville and Soundararajan \cite[Lemma 8.2]{GS2001})}\label{Lemma_GS}
  Let $s=\sigma+it$ with $\sigma>1/2$ and $|t|\leq 2q$. Let $y\geq 2$ be
  a real number and let $1/2\leq \sigma_0<\sigma$.
  Suppose that there are no zeros of $L(z,\chi)$ inside the rectangle
  $\{z \;:\; \sigma_0\leq \Re(z) \leq 1,\; |\Im(z)-t|\leq y+3\}$.
  Put $\sigma_1=\{(\sigma+\sigma_0)/2,\;\sigma_0+(\log y)^{-1}\}$.
  Then
  \[
  \log L(z,\chi)=\sum_{n=2}^y\frac{\Lambda (n)\chi(n)}{n^s\log n}
  +O\bigg(\frac{\log q}{(\sigma_1-\sigma_0)^2}y^{\sigma_1-\sigma}\bigg).
  \]
\end{prop}
\par
The above results have been generalized to other $L$-functions.
For example, 
Lau and Wu \cite{LW} obtained a result similar to Elliott's one for symmetric power
$L$-functions.

Also in the case of symmetric power $L$-functions,
in \cite{MU2018}, we have an asymptotic formula
for the values of logarithms of them,
which is proved by applying the method in the proof of Duke \cite[Lemma~3]{duke}.

\par
In this paper, we apply Duke's argument to general $L$-functions and
obtain an asymptotic formula for the values of logarithms
of general $L$-functions at and near $s=1$.
Our result includes a new asymptotic formula in the case of
Dirichlet $L$-functions.
\bigskip

{\bf Acknowledgment}
The authors express their thanks to Professor A. Sankaranarayanan, whose comment
on our previous article \cite{MU2018} is one of the starting points of the
present research.

\section{General $L$-functions and their truncations}
\par
In this section we define the general form of $L$-functions which we will consider
in the present paper.
\par
Let $m$ be a fixed positive integer.
Let $s=\sigma + it$ $(\sigma, t\in\mathbb{R})$.
$L(s)$ is a $L$-function
defined by a Euler product
\begin{equation}\label{def}
L(s)
=\prod_{p\text{: prime}}\prod_{j=1}^{m}(1-c_j(p)p^{-s})^{-1}
\end{equation}
with $|c_j(p)| \leq 1$ for $\sigma>1$.
We know  $L(1+\varepsilon+it)\ll_{\varepsilon} 1$ for any $\varepsilon>0$.
\par
We assume the following assumptions.
%
%
\begin{ass}\label{conti-analy}
  $L(s)$ is continued analytically as a holomorphic function for
  $\sigma\geq \sigma_L$ $(1>\sigma_L\geq 1/2)$.
\end{ass}
%
%
\begin{ass}\label{zero-free}
Let $N$ be an intger with $\log N> e$
(which is equivalent to $N\geq 16$).
  There exists $0<\Delta<7(1-\sigma_L)\leq 7/2$ such that
  $L(s)$ has no zero in $D_{\Delta, N}$, where 
  \begin{align*}
    D_{\Delta, N}
    =&
    D_1
    \cup
    \{s\in\mathbb{C} \mid 1-\Delta/7 \leq \sigma <1, \; |t| \leq 2(\log N)^3\},
    \\
    D_1=& \{s\in\mathbb{C} \mid \sigma \geq 1\}.
  \end{align*}
  We note that $1/2\leq \sigma_L<1-\Delta/7$ and $0<\Delta<7/2$.
\end{ass}
%
%
\begin{ass}\label{convex}
  $L(s)$ has a bound
  \[
  L(s) \ll N^{\gamma_1}(1+|t|)^{\gamma_2}
  \]
  for $\sigma_L <\sigma < 4$, where $\gamma_1 >0$ and $\gamma_2>0$ are fixed real numbers.
\end{ass}

Many type of $L$-functions satisfy
Assumptions~\ref{conti-analy} and \ref{convex}.
In typical cases, $N$ can be a conductor of characters
or a level of newforms, etc.
\par
In this paper, the implied constants of $\ll$ depend on fixed constants
$m$, $\gamma_1$ and $\gamma_2$.

Next we define the finite truncation of our $L$-function.
We denote a partial $L$-function
\[
L_{p \leq (\log N)^A}(s)
=\prod_{p\leq (\log N)^A}\prod_{j=1}^{m}(1-c_j(p)p^{-s})^{-1}
\]
for $A > 0$.
Then $L_{p\leq (\log N)^A}(s)$ is defined for $\sigma >0$ and 
$L_{p\leq (\log N)^A}(s)\neq 0$. 
Let
$$
F(s)=\frac{L(s)}{L_{p\leq (\log N)^A}(s)}.
$$
Then, under the above assumptions, we may define
\[
\log F(s)=\log L(s)-\log L_{p\leq (\log N)^A}(s),
\]
which is holomorphic in $D_{\Delta, N}$.
For $\sigma >1$, we know 
\[
  \log F(s)
  =
  -\sum_{p>(\log N)^A}\sum_{j=1}^m\log (1-c_j(p)p^{-s})
  =
  \sum_{p>(\log N)^A}\sum_{j=1}^m\sum_{\ell=1}^{\infty}
    \frac{{c_j}^{\ell}(p)}{\ell p^{\ell s}}.
\]
Putting
\[
d_p(\ell) = \sum_{j=1}^m c_j^{\ell}(p)
\]
we see that 
$|d_p(\ell)| \leq m$, and
\begin{align*}
  \log F(s)
  =&
  \sum_{p>(\log N)^A}
  \sum_{\ell=1}^{\infty}
  \frac{d_p(\ell)}{\ell p^{\ell s}},
  \\
  \frac{F'}{F}(s)
  =&
  -
  \sum_{p>(\log N)^A}
  \sum_{\ell=1}^{\infty}
  \frac{d_p(\ell)\log p}{p^{\ell s}}.
\end{align*}
\section{Statement of results}

Now we state our main results.    The first result is as follows.

\begin{thm}\label{thm-1}
  Let $N\geq 16$ be a positive number.
  Under Assumptions~\ref{conti-analy}, \ref{zero-free} and \ref{convex},
  for $7/2 > \delta > 0$ and $A\geq 14/\delta$,
  we have
  \[
    L(1)=\bigg(1+O_{\delta, A}\bigg(\frac{1}{\log\log N}\bigg)\bigg)\prod_{p\leq (\log N)^A}\prod_{j=1}^m(1-c_j(p)p^{-1})^{-1}.
  \]
\end{thm}

\begin{cor}\label{cor-1}
  Let $Q\geq 16$ be a positive integer.
  Let $\chi$ be a primitive Dirichlet character of conductor $q\leq Q$.
  Then, for $7/2 > \Delta > 0$ and $A\geq 14/\delta$,
  we have
  \[
    L(1,\chi)=\bigg(1+O_{\delta, A}\bigg(\frac{1}{\log\log Q}\bigg)\bigg)\prod_{p\leq (\log Q)^A}(1-\chi(p)p^{-1})^{-1},
  \]
  unless $\chi$ lies in an exceptional set $E(\delta)$.
  Here $E(\delta)$ contains  
  $\ll_{\Delta} Q^{\Delta}$ primitive characters with conductor $q\leq Q$.
\end{cor}
\begin{proof}
  Dirichlet $L$-functions satisfy Assumptions~\ref{conti-analy} and \ref{convex},
  with $N=Q$.
  At the end of the proof of Lemma~2 in Montgomery and Weinberger~\cite{MW},
  it is shown that the number of primitive characters which do not satisfy
  Assumption~\ref{zero-free} with $N=Q$ is
  $ \ll_{\Delta} Q^{\Delta}$
  (see \eqref{estimate_E}).
  Therefore Corollary \ref{cor-1} follows from Theorem \ref{thm-1}.
\end{proof}
\begin{rem}
  Comparing this corollary with Proposition \ref{lem:E} (with $A=14/\delta$), we see that
  the form of the formula in Corollary \ref{cor-1} is not so sharp as in the formula of
  Proposition~\ref{lem:E}.
  However, in Proposition \ref{lem:E} there is the condition $A\leq\log\log Q$,
  which implies $\delta\geq 14/\log\log Q$.
  In our corollary there is no such restriction on the range of $\delta$.
  \end{rem}

As for $\log L(1)$, we obtain the following theorems.
For $\alpha >1/2$, we define
\[
C(\alpha)=\sum_p\sum_{\ell=2}^{\infty}\frac{d_p(\ell)}{\ell p^{\ell\alpha}}.
\]
\begin{thm}\label{thm-2}
  Let $N\geq 16$ be a large positive integer.
  Under Assumptions~\ref{conti-analy}, \ref{zero-free} and \ref{convex},
  for $7/2 > \Delta > 0$,
  we have
  \[
  \log L(1)
  =
  \sum_{p\leq (\log N)^{14/\Delta}}\frac{d_p(1)}{p}
  +
  C(1)
  +
  O_{\delta}\bigg(
    \frac{1}{\log\log N}
    \bigg).
  \]
\end{thm}

\begin{thm}\label{thm:main_p_s=1}
  Let $N\geq 16$ be a positive integer
  and $1-\delta/28 < \alpha < 2-\delta/7$.
  Under Assumption~\ref{conti-analy}, \ref{zero-free} and \ref{convex},
  for $7/2>\Delta>0$ and for $(\log N)^{14/\delta}\leq y \leq N^{2\pi/5}$,
  we have
\begin{align*}
  &\log L(\alpha)
  =
  \sum_{p \leq y} \frac{d_p(1)}{p^{\alpha}} +C(\alpha)
  \nonumber\\
  &+ 
  O\bigg(
    \frac{y^{1-\alpha}}{(2-\alpha)^2\log y} + 
    \frac{(\log N)^2}{(\log\log N)(\log N)^{14/\delta}}\nonumber\\
   &\qquad+
  \frac{y^{1+\delta/14}(\log N)^{\delta/14+3/2}}{\delta^2N^{\pi/2}}
  +
  \frac{y^{1-\delta/14-\alpha}(\log N)}{\delta^3 \log y}
%
  \bigg).
  \end{align*}
Especially, we have
\begin{align}\label{formula:main_p_s=1}
    &\log L(1)
  =
  \sum_{p \leq y} \frac{d_p(1)}{p} +C(1)
  \nonumber\\
  &+ 
  O\bigg(
    \frac{1}{\log y} + \frac{(\log N)^2}{(\log\log N)(\log N)^{14/\delta}}
  +
  \frac{y^{1+\delta/14}(\log N)^{\delta/14+3/2}}{\delta^2N^{\pi/2}}
  +
  \frac{\log N}{\delta^3 y^{\delta/14} \log y}
  \bigg).
\end{align}
\end{thm}
\begin{rem}\label{rem:2_ijou}
In Theorem \ref{thm:main_p_s=1}, we require the condition $\alpha < 2-\delta/7$ on
the region of $\alpha$.    It is possible to extend the result to the case of larger
$\alpha$, but then, the statement should be modified (see Remark \ref{rem:2_ika}).
In any case, such an extension is not important for the main aim of the present paper.
\end{rem}
\begin{rem}\label{rem:thm3}
The assumption $(\log N)^{14/\delta}\leq y \leq N^{2\pi/5}$ ensures that the error term
on the right-hand side of \eqref{formula:main_p_s=1} is indeed $o(1)$ as $N\to\infty$.
In fact, since $1+\delta/14<5/4$, we see that
$ y^{1+\delta/14}(\log N)^{\delta/14+3/2}=o(N^{\pi/2})$.   Also
$y^{-\delta/14}\log N/\log y\ll (\log\log N)^{-1}$.
\end{rem}
\begin{rem}
The case $y=(\log N)^{14/\delta}$ in Theorem~\ref{thm:main_p_s=1} gives a formula
whose error estimate is better than that in Theorem~\ref{thm-2}.
However, the assumption $(\log N)^{14/\delta}\leq y \leq N^{2\pi/5}$ in Theorem
\ref{thm:main_p_s=1} implies the restriction
$(14/\delta)(\log\log N)\leq (2\pi/5)(\log N)$, hence
$\delta\geq (35\log\log N)/(\pi\log N)$.
Theorem~\ref{thm-2} can be applied to arbitrarily small $\delta$.
\end{rem}
\begin{cor}\label{cor-1.5}
    Let $N\geq 16$ be a positive integer.
  Under Assumptions~\ref{conti-analy}, \ref{zero-free} and \ref{convex},
  for $7/2>\Delta>0$.
  Let $a$ be real, satisfying 
  $\delta\log y > 28a \geq 0$
  for $(\log N)^{14/\delta}\leq y \leq N^{2\pi/5}$.
  We have
\begin{align*}
  &\log L(1-a/\log y)
  =
  \sum_{p \leq  y} \frac{d_p(1)}{p^{1-a/\log y}} +C(1-a/\log y)
  \nonumber\\
  &+ 
  O\bigg(
    \frac{e^a}{\log y} + \frac{(\log N)^2}{(\log\log N)(\log N)^{14/\delta}}
  +
  \frac{y^{1+\delta/14}(\log N)^{\delta/14+3/2}}{\delta^2N^{\pi/2}}
  +
  \frac{e^a (\log N)}{\delta^3 y^{\delta/14} \log y}
  \bigg).
\end{align*}
\end{cor}
\begin{proof}
  In Theorem~\ref{thm:main_p_s=1}, we put $\alpha=1-a/\log y$.
\end{proof}
\begin{rem}
  A real number $a\geq 0$ with $2a <  \log\log N$ satisfies
  $28 a<\delta \log y$ with $(\log N)^{14/\delta}<y$.
  If $a<(1-\varepsilon)\log\log\log N$ for some $\varepsilon>0$, 
  the error term on the right-hand side of the formula in this corollary
  is $o(1)$ as $N\to\infty$ by the same reason as in Remark~\ref{rem:thm3}.
\end{rem}
Applying the above results to Dirichlet $L$-functions, we have

\begin{cor}\label{cor-2}
  Let $Q\geq 16$, $7/2>\Delta>0$, 
  and $(\log Q)^{14/\delta}\leq y \leq Q^{2\pi/5}$.
  For a real number $a$ with $\delta\log y > 28 a \geq 0$
  and a primitive Dirichlet character $\chi \pmod{q}$ with $q\leq Q$,
  we have
  \begin{align}\label{cor-2-1}
    &
    \log L(1-a/\log y, \chi)
    \nonumber\\
  =&
  \sum_{p\leq y}\frac{\chi (p)}{p^{1-a/\log y}}
  +
  C(1-a/\log y,\chi)
  \nonumber\\
  &+
  O\bigg(\frac{e^a}{\log y} +\frac{(\log Q)^2}{(\log\log Q)(\log Q)^{14/\delta}}+\frac{y^{1+\delta/14}(\log Q)^{\delta/14+3/2}}{\delta^2 Q^{\pi/2}}+\frac{e^a (\log Q)}{\delta^3y^{\delta/14}\log y} \bigg),
\end{align}
moreover
\begin{align}\label{formula:main_p_s=1:Dirichelt}
  &\log L(1,\chi)
  =
  \sum_{p <  (\log Q)^{14/\delta}} \frac{\chi(p)}{p} +C(1,\chi)
  +
  O_{\delta}\bigg(
  \frac{1}{\log\log Q}
  \bigg),
  \end{align}
unless $\chi$ lies in an exceptional set $E(\Delta)$
which is the same as in the statement of Corollary~\ref{cor-1}.
\end{cor}
\begin{proof}
Similar to that of Corollary~\ref{cor-1}.
\end{proof}
\begin{rem}
This corollary should be compared with \eqref{MV}.   The implied constant in
\eqref{MV} depends on $\chi$, but the implied constant in
\eqref{cor-2-1}
depends only on $\delta$.
\end{rem}

In the case of Dirichlet $L$-functions, we may proceed further, because we know more
information on these functions.

For example, let $\chi$ be a primitive character mod $q$, and
let $N(T,\chi)$ be the number of non-trivial zeros $\rho=\beta+i\gamma$ of $L(s,\chi)$
in the region $0\leq \beta \leq 1$ and $|\gamma|\leq T$.   Then it is known that
\begin{align}\label{DL1}
N(T,\chi)=\frac{T}{\pi}\log\frac{qT}{2\pi}-\frac{T}{\pi}+O(\log T+\log q)
\end{align}
for $2\leq T$ (see \cite[Chapter 16]{davenport}).
Define
\[
\psi(x,\chi)
=\sum_{n\leq x}\Lambda(n)\chi(n).
\]
Then
\begin{align}\label{DL2}
\psi(x,\chi)
=\sum_{p\leq x}\chi(p)\log p+ O(\sqrt{x}\log x).
\end{align}
In the course of the proof of the prime number theorem in arithmetic progressions 
(see \cite[Chapter 19]{davenport}), it is shown that
\begin{align}\label{DL3}
\psi(x,\chi)=
-\frac{x^{\beta_1}}{\beta_1}-\sum_{|\gamma|<T}\frac{x^{\rho}}{\rho}+R_3(x,T)
\end{align}
for $2\leq T \leq x$, where $\beta_1$ is (possible) Siegel's zero, 
$\rho$ are other non-trivial zeros of $L(s,\chi)$ (except $1-\beta_1$), and
\begin{align}\label{DL4}
R_3(x,T) \ll \frac{x(\log qx)^2}{T}+x^{1/4}\log x.
\end{align}
Using these formulas, we can obtain the following refinements of our previous results.

\begin{thm}\label{cor:apply_PNT}
 Let $Q\geq 16$ and $7/2>\delta>0$ and $a$ be real with $\delta\log Q > 28a\geq 0$.
 Let $\chi$ be a primitive Dirichlet character of conductor $q\leq Q$.
 We have
  \begin{align}\label{formula_th4_1}
    \log L(1-a/\log Q,\chi)
    = &
    \sum_{p\leq (\log Q)^{14/\Delta}}\frac{\chi(p)}{p^{1-a/\log Q}} +C(1-a/\log Q,\chi)
    \nonumber\\
    &+
    O_{\Delta}\left( 
    \frac{e^a \log\log Q}{\log Q}\right),
  \end{align}
  unless $\chi$ lies in the exceptional set $E(\Delta)$.
  This further yields that
\begin{align}\label{formula_th4_2}
  L(1-a/\log Q,\chi)
  =\bigg(
  1+O_{\Delta}\bigg(\frac{e^a \log\log Q}{\log Q}\bigg)
  \bigg)
  \prod_{p\leq (\log Q)^{14/\Delta}}
  \bigg(1-\frac{\chi(p)}{p^{1-a/\log Q}}\bigg)^{-1}
\end{align}
  unless $\chi \in E(\Delta)$.
\end{thm}
\begin{rem}
The error term on the right-hand side of \eqref{formula_th4_2} is better than the
corresponding error estimete in the second formula of Proposition \ref{lem:E}.
We can also see that \eqref{formula_th4_1} essentially implies Proposition
\ref{lem:MW} (see Remark \ref{4implies1}).
\end{rem}
It is possible to obtain a formula, similar to \eqref{formula_th4_1} but for shorter sum,
by appealing the large sieve inequality:
\begin{thm}\label{cor:apply_PNT2}
  Let $0<\delta<2$, 
  and $a$ be real with $\delta\log Q > 28a\geq 0$,
  $28a/(\delta\log Q) <A<14/\delta$.
For almost all primitive characters $\chi$ mod $q\leq Q$
and a real number $a$ with $\delta\log Q > 28a \geq 0$, we have
  \begin{align*}
    &
    \log L(1-a/\log Q, \chi)
    \\
    =&
    \sum_{p\leq (\log Q)^A}\frac{\chi(p)}{p^{1-a/\log Q}} +C(1-a/\log Q,\chi)
 +O_{\Delta}\left(
     \frac{e^a\log\log Q}{\log Q}+\frac{f(Q)}{(\log Q)^{A/2}}\right),
  \end{align*}
  where $f(Q)$ is a function in $Q$ which tends to infinity, satisfies 
  $(\log Q)^{14a/(\delta\log Q)} =o(f(Q))$ and
  $f(Q)=o((\log Q)^{A/2})$, as $Q\to\infty$.
  Especially we have
  \begin{align*}
    &\log L(1,\chi)= \sum_{p\leq (\log Q)^A}\frac{\chi(p)}{p} +C(1,\chi)
 +O_{\Delta}\left(
     \frac{e^a\log\log Q}{\log Q}+\frac{f(Q)}{(\log Q)^{A/2}}\right),
  \end{align*}
  where $f(Q)$ stands for any function in $Q$ which tends to infinity
  and satisfies $f(Q)=o((\log Q)^{A/2})$
  as $Q\to\infty$.
\end{thm}
\begin{rem}
We state Theorem \ref{cor:apply_PNT} and Theorem \ref{cor:apply_PNT2} 
(and Corollary \ref{cor-1}, Corollary \ref{cor-2})
only for Dirichlet $L$-functions.
However, to show those results we only use the zero density estimate
\eqref{zero-density} and the explicit formula \eqref{DL3}, \eqref{DL4}.
Therefore it is not difficult to generalize those results to a more
general class of $L$-functions, for which such kind of zero-density estimate
and explicit formula can be proved (see, for instance, Perelli \cite{Pe}).
\end{rem}

The rest of the present paper is devoted to the proof of the above theorems.
After the preparations in the next two sections, we will prove
Theorems \ref{thm-1} and \ref{thm-2} in Section \ref{sec-6}, 
Theorem \ref{thm:main_p_s=1} in Section \ref{sec-7}, and Theorems
\ref{cor:apply_PNT} and \ref{cor:apply_PNT2} in Section \ref{sec-8}.

\section{Estimations of logarithmic derivatives}\label{sec-4}
\par
Now we start the proof of theorems.    First of all, in this section we prepare some
upper bound estimates of $(L'/L)(s)$ and 
$(L'_{p\leq (\log N)^A}/L_{p\leq (\log N)^A})(s)$.
\par
Let $7/2>\delta>0$, $N\geq 16$.   For $t_0 \in\mathbb{R}$,
define $s_0\in D_{\delta,N}$ by 
\[
s_0 =
\begin{cases}
  2+it_0 & |t_0|\leq \log N \\
  2+\Delta/7+it_0 & |t_0|> \log N, 
\end{cases}
\]
and let $r=1+\Delta/7$.
We denote the domain
\[
D_{\Delta, N}^0= \bigcup_{t_0\in\mathbb{R}}\{s\in\mathbb{C} \mid |s-s_0|<r\}.
\]
We see that $D_{\Delta, N}^0 \subset D_{\Delta, N}$.
In $D_{\Delta, N}^0$, $\log L(s)$ is a holomorphic function
because of Assumption~\ref{zero-free}.
\par
For $|s-s_0|<r$, since $\sigma_L<1-\Delta/7< \sigma <3+2\Delta/7$,
we find that
\begin{align*}
  \Re(\log L(s) -\log L(s_0))
  =&
  \log |L(s)| -\log |L(s_0)|
  \\
  \ll&
  \log (N^{\gamma_1}(1+|t|)^{\gamma_2}) + \log (N^{\gamma_1}(1+|t_0|)^{\gamma_2})
  \\
  \ll&
  \log N + \log(1+|t|) + \log (1+|t|+r)
  \\
  \ll&
  \log (N(1+|t|))
\end{align*}
from Assumption~\ref{convex}.
By Lemma~2 in Duke \cite{duke} (which is Lemma~4 in Littlewood \cite{L}), 
there exists an positive constant $A$ depending on $\gamma_1$ and $\gamma_2$
so that for $|s-s_0|\leq r_1 < r$ we have
\[
\left|\frac{L'(s)}{L(s)}\right|
<
\frac{Ar\log (N(1+|t|))}{(r-r_1)^2}
\ll
\frac{\log (N(1+|t|))}{(r-r_1)^2}.
\]
Choosing 
$r_1=1+\delta/14$, we have
\begin{equation}\label{diff_logL}
  \left|\frac{L'(s)}{L(s)}\right| \ll \frac{\log (N(1+|t|))}{\delta^2}
\end{equation}
for all $s\in D_{\delta, N}^1$, where
\[
D_{\delta, N}^1
=
\bigcup_{t_0\in\mathbb{R}}\{s\in\mathbb{C} \mid |s-s_0| \leq r_1\}
\subset D_{\Delta, N}^0.
\]
\par
Next we consider the logarithmic derivative of $L_{p\leq (\log N)^A} (s)$.
For $s\in D_{\delta, N}^1$,
we know that $\log L_{p\leq (\log N)^A}(s)$ is holomorphic and
\begin{align*}
  \log L_{p\leq (\log N)^A} (s)
  =
  -\sum_{p\leq (\log N)^A}\sum_{j=1}^{m}\log (1-c_j(p)p^{-s})&
  \\
  =
  \sum_{p\leq (\log N)^A}\sum_{j=1}^{m}\sum_{\ell=1}^{\infty}
  \frac{c_j^{\ell}(p)}{\ell p^{\ell s}}
  =
  \sum_{p\leq (\log N)^A}\sum_{\ell=1}^{\infty}\frac{d_p(\ell)}{\ell p^{\ell s}}.&
  \\
\end{align*}
Therefore, using $d_p(\ell)\ll 1$ and $\sigma\geq 1-\delta/14>1/2$
for $s\in D_{\delta, N}^1$, we have
\begin{align*}
 \lefteqn{ \frac{L'_{p\leq (\log N)^A}}{L_{p\leq (\log N)^A}}(s)
  =
  -
  \sum_{p\leq (\log N)^A}\sum_{\ell=1}^{\infty}\frac{d_p(\ell)\log p}{p^{\ell s}}}
  \nonumber\\
 & \ll
  \sum_{p\leq (\log N)^A}\frac{\log p}{p^{\sigma}}
  +
  \sum_{p\leq (\log N)^A}\frac{\log p}{p^{2\sigma}}
  +
  \sum_{p\leq (\log N)^A}\sum_{\ell=3}^{\infty}\frac{\log p}{p^{\ell\sigma}}
  \nonumber\\
 & \ll
  \sum_{p\leq (\log N)^A}\frac{\log p}{p^{\sigma}}
  +
  \sum_{p\leq (\log N)^A}\frac{\log p}{p^{3\sigma}}.
  \nonumber
\end{align*}
Therefore we obtain
\begin{align}\label{diff_logpartL}
\lefteqn{ \frac{L'_{p\leq (\log N)^A}}{L_{p\leq (\log N)^A}}(s)}\nonumber\\
   &\ll
  \begin{cases}
    \min\{ (\log N)^{A(1-\sigma)}/(1-\sigma),\; (\log N)^{A(1-\sigma)}(A\log\log N)\} & 
    1> \sigma \geq 1-\delta/14 \\
    \displaystyle A\log\log N & \sigma=1 \\
    1 & \sigma > 1
  \end{cases}
\end{align}
for any $s\in D_{\delta, N}^1$,
because for $\sigma<1$ we know
\begin{align*}
\sum_{p\leq (\log N)^A}\frac{\log p}{p^{\sigma}}
&=
\frac{1}{(\log N)^{A\sigma}}\sum_{p\leq (\log N)^A} \log p
+\sigma \int_2^{(\log N)^A}t^{-\sigma-1}O(t) dt\\
&\ll
\frac{(\log N)^{A(1-\sigma)}}{1-\sigma}
\end{align*}
by the fact $\sum_{p\leq x}\log p \ll x$, 
and
\begin{align*}
\sum_{p\leq (\log N)^A}\frac{\log p}{p^{\sigma}}
=&
(\log N)^{A(1-\sigma)}\sum_{p\leq (\log N)^A}\frac{\log p}{p}
-(1-\sigma)\int_2^{(\log N)^A} \!\!\! t^{-\sigma} O(\log t) dt
\\
\ll&
(\log N)^{A(1-\sigma)}(A\log\log N)
\end{align*}
by the fact
\[
\sum_{p\leq x}\frac{\log p}{p} \ll \log x.
\]
\section{An auxiliary formula for $\log F(\alpha)$}\label{sec-5}
\par
Recall
\[
\log F (s)
  =
  \log L(s)-\log L_{p\leq (\log N)^A}(s).
\]
The aim of this section is to prove the following formula.

\begin{lemm}\label{aux}
Let $1-\delta/28< \alpha < 2-\delta/7$ and $\beta=1-\delta/14$ (hence $\beta<\alpha$).
For $v>1$ and $0< A \leq B$, we have
\begin{align}\label{L_alpha}
  &
  -\sum_{(\log N)^B \geq p>(\log N)^A} \frac{d_p(1) e^{-p/v}}{p^{\alpha}}
  +
  \log F(\alpha)
  \nonumber\\
  \ll&
 \frac{v}{(\alpha B\log\log N)(\log N)^{\alpha B}}
  +
  \frac{(\log N)^{A\delta/7}}{(A\log\log N)(\log N)^A}
  \nonumber\\
  &+
\bigg(\frac{\log N}{\delta^2}+A(\log N)^{A\delta/14}
  +\frac{(\log N)^{A\delta/14}}{\delta}\bigg)
 \bigg(\frac{v^{1+\delta/14}(\log N)^{\delta/14+1/2}}{N^{\pi/2}}
+
\frac{v^{\beta-\alpha}}{\delta \log v}
\bigg)
\end{align}
\end{lemm} 
\begin{proof}
Let $v>1$, 
and $\alpha \leq u < 2-\delta/7$.
Mellin's formula
\[
e^{-w}=\frac{1}{2\pi i}\int_{(1+\delta/14)}w^{-z}\Gamma(z)dz
\]
(where the path of integration is the vertical line $\Re z=1+\delta/14$)
yields that
\begin{align*}
  &
  -\sum_{p>(\log N)^A}\sum_{\ell=1}^{\infty}\frac{d_p(\ell)\log p}{p^{\ell u}}e^{-p^{\ell}/v}
  \\
  =&
  -\frac{1}{2\pi i}\int_{(1+\delta/14)}
  \sum_{p>(\log N)^A}\sum_{\ell=1}^{\infty}\frac{d_p(\ell)\log p}{p^{\ell (u+z)}}
  v^z\Gamma(z)dz
  \\
  =&
  \frac{1}{2\pi i}\int_{(1+\delta/14)}
  \frac{F'}{F}(u+z)
  v^z\Gamma(z)dz,
\end{align*}
where $A>0$.
Noting 
\begin{align}\label{futousiki}
0>-\delta/28 > \beta-u > -1+\delta/14 >-1,
\end{align}
we shift the path of integration to $\Re z=\beta-u$ and apply
the residue theorem to obtain
\begin{align}\label{shifting}
  &
  -\sum_{p>(\log N)^A}\sum_{\ell=1}^{\infty}
  \frac{d_p(\ell)\log p}{p^{\ell u}}e^{-p^{\ell}/v}\notag
  \\
  =&
  \frac{F'}{F}(u)
  +
  \frac{1}{2\pi i}\int_{\substack{\beta-u\leq x \leq 1+\delta/14\\|y|=\log N}}
  \frac{F'}{F}(u+z)v^z\Gamma(z)dz\notag
  \\
  &
  +
  \frac{1}{2\pi i}\int_{\substack{x=\beta-u\\|y|\leq \log N}}
  \frac{F'}{F}(u+z)v^z\Gamma(z)dz
  +
  \frac{1}{2\pi i}\int_{\substack{x = 1+\delta/14\\|y|\geq \log N}}
  \frac{F'}{F}(u+z)v^z\Gamma(z)dz,
\end{align}
where $z=x+iy$, because by Assumption \ref{zero-free} the only relevant pole is
$z=0$.
We write the right-hand side as
$$
  \frac{F'}{F}(u)+I_1+I_2+I_3.
$$
We remark that $z+u\in D_{\delta, N}^1$,
when $z$ is on the above path of integration.
Hence we can estimate the above integrals by
\eqref{diff_logL} and \eqref{diff_logpartL}.
%
%
We see 
that
\begin{align*}
  I_1
  \ll 
  \bigg(\frac{\log N}{\delta^2} +(\log N)^{A\delta/14}(A\log\log N)\bigg)
  \int_{\beta-u\leq x \leq 1+\delta/14} v^x|\Gamma(x\pm i\log N)|dx.
\end{align*}
The estimate
\begin{equation}\label{sterling_large_t}
  \Gamma(s) \ll |t|^{\sigma-1/2}e^{-\pi|t|/2}
\end{equation}
in the region $|\sigma|<c$ and $|t|\geq 1$ ($c$ is an absolute constant)
is well-known.
Applying \eqref{sterling_large_t} we can estimate the last integral as
\begin{align*}
  \ll&
  \int_{\beta-u\leq x \leq 1+\delta/14} v^x(\log N)^{x-1/2}e^{-\pi(\log N)/2} dx
  \\
  \ll&
  N^{-\pi/2}(\log N)^{-1/2}
  \bigg[\frac{v^x(\log N)^x}{\log v+\log\log N}\bigg]_{\beta-u}^{1+\delta/14}
  \\
  \ll&
  N^{-\pi/2}(\log N)^{-1/2}
  \frac{v^{1+\delta/14}(\log N)^{1+\delta/14}}{\log v+\log\log N},
\end{align*}
because 
$v\log N >1$.
Therefore we obtain that
\begin{equation}\label{error_int1}
  I_1
  \ll 
  \bigg(\frac{\log N}{\delta^2} +(\log N)^{A\delta/14}(A\log\log N)\bigg)
  \frac{v^{1+\delta/14}(\log N)^{1/2+\delta/14}}{N^{\pi/2}(\log v+\log\log N)}.
\end{equation}
Next, from \eqref{diff_logL}, \eqref{diff_logpartL}
and \eqref{sterling_large_t},
we see that
\begin{align}\label{error_int2}
  I_2
  =&
  \frac{1}{2\pi i}\int_{\substack{x=\beta-u\\|y|\leq \log N}}
  \frac{F'}{F}(u+z)v^z\Gamma(z)dz
  \nonumber\\
  \ll &
  \int_{\substack{x=\beta-u\\|y|\leq \log N}}
  \bigg(\frac{\log N}{\delta^2}+\frac{(\log N)^{A\delta/14}}{\delta}\bigg)
  v^{\beta-u}|\Gamma(z)|dz
  \nonumber\\
  \ll &
  \bigg(\frac{\log N}{\delta^2}+\frac{(\log N)^{A\delta/14}}{\delta}\bigg)
  v^{\beta-u}
  \nonumber\\
  &\times
  \bigg(\int_{\substack{x=\beta-u\\|y|\leq 1}}
  |\Gamma(z)| dy
  +
  \int_{1 \leq |y|\leq \log N}
  |y|^{\beta-u-1/2}e^{-\pi|y|/2}dy
  \bigg)
  \nonumber\\
  \ll &
  \bigg(\frac{\log N}{\delta^2}+\frac{(\log N)^{A\delta/14}}{\delta}\bigg)
  v^{\beta-u}
  \left(\frac{1}{\delta}+1\right)
  \nonumber\\
  \ll &
  \bigg(\frac{\log N}{\delta^2}+\frac{(\log N)^{A\delta/14}}{\delta}\bigg)
  \frac{v^{\beta-u}}{\delta},
\end{align}
since $-\delta/28 > \beta-u > -1+\delta/14$.
Moreover, by 
\eqref{sterling_large_t}, we see that
\begin{align}\label{error_int3}
  I_3
  =&
  \frac{1}{2\pi i}\int_{\substack{x = 1+\delta/14\\|y|\geq \log N}}
  \frac{F'}{F}(u+z)v^z\Gamma(z)dz
  \nonumber\\
  \ll &
  \int_{\substack{x = 1+\delta/14 \\y\geq \log N}}
  v^{1+\delta/14}y^{1/2+\delta/14}e^{-\pi y/2}dy
  \nonumber\\
  \ll&
  v^{1+\delta/14}
  \int_{y\geq \log N} ye^{-\pi y/2}dy
  \nonumber\\
  \ll &
  \frac{(\log N)v^{1+\delta/14}}{N^{\pi/2}}.
\end{align}
From \eqref{error_int1}, \eqref{error_int2} and \eqref{error_int3},
we obtain that
\begin{align*}
  &
  -\sum_{p>(\log N)^A}\sum_{\ell=1}^{\infty}
  \frac{d_p(\ell)\log p}{p^{\ell u}}e^{-p^{\ell}/v}
  \\
  =&
  \frac{F'}{F}(u)
  +
  O\bigg(
  \bigg(\frac{\log N}{\delta^2}+(\log N)^{A\delta/14}(A\log\log N)\bigg)
  \frac{v^{1+\delta/14}(\log N)^{1/2+\delta/14}}{N^{\pi/2}(\log v +\log\log N)}
  \bigg)
  \\
  &+
  O\bigg(
  \bigg(\frac{\log N}{\delta^2} + \frac{(\log N)^{A\delta/14}}{\delta}\bigg)
  \frac{v^{\beta-u}}{\delta}
  +
  \frac{v^{1+\delta/14}\log N}{ N^{\pi/2}}
  \bigg)
  \\
  =&
  \frac{F'}{F}(u)
  +
  O\bigg(
  \frac{Av^{1+\delta/14}(\log N)^{1/2+\delta/14+A\delta/14}}{N^{\pi/2}}
  \bigg)
  \\
  &+
  O\bigg(
  \bigg(\frac{\log N}{\delta^2} + \frac{(\log N)^{A\delta/14}}{\delta}\bigg)
  \frac{v^{\beta-u}}{\delta}
  +
  \frac{v^{1+\delta/14}(\log N)^{3/2+\delta/14}}{\delta^2 N^{\pi/2}}
  \bigg).
\end{align*}
\par
We integrate the both sides with respect to $u$ from $\alpha$ to $2-\delta/7$ to obtain
\begin{align*}
  &
  -\sum_{p>(\log N)^A}\sum_{\ell=1}^{\infty}d_p(\ell)(\log p) e^{-p^{\ell}/v}
  \int_{\alpha}^{2-\delta/7}p^{-\ell u}du
  \\
  =&
  \int_{\alpha}^{2-\delta/7}\frac{F'}{F}(u)du
  +
  O\bigg(
  \frac{Av^{1+\delta/14}(\log N)^{1/2+\delta/14+A\delta/14}}{N^{\pi/2}}
  \bigg)
  \\
  &+
  O\bigg(
  \bigg(\frac{\log N}{\delta^2} + \frac{(\log N)^{A\delta/14}}{\delta}\bigg)
  \frac{v^{\beta}}{\delta}\int_{\alpha}^{2-\delta/7} v^{-u} du
  +
  \frac{v^{1+\delta/14}(\log N)^{3/2+\delta/14}}{ N^{\pi/2}}
  \bigg).
\end{align*}
The right-hand side is
\begin{align*}
  =&
   \log F(2-\delta/7) -\log F(\alpha)
  \\
  &+
  O\bigg(
    \frac{Av^{1+\delta/14}(\log N)^{1/2+\delta/14+A\delta/14}}{N^{\pi/2}}
    + \frac{v^{1+\delta/14}(\log N)^{3/2+\delta/14}}{N^{\pi/2}}
    \bigg)
  \\
   & +
  O\bigg(
  \bigg(\frac{\log N}{\delta^2} + \frac{(\log N)^{A\delta/14}}{\delta}\bigg)
  \frac{v^{\beta}}{\delta}
  \bigg[\frac{-v^{-u}}{\log v}\bigg]_{\alpha}^{2-\delta/7}
  \bigg)
  \\
  =&
  \log F(2-\delta/7) -\log F(\alpha)
  \\
  &+
  O\bigg(
    \bigg(A(\log N)^{A\delta/14}
    + \log N
    \bigg)\frac{v^{1+\delta/14}(\log N)^{\delta/14+1/2}}{N^{\pi/2}}
    \bigg)
    \\
    & +
    O\bigg(
  \bigg(\frac{\log N}{\delta^2} + \frac{(\log N)^{A\delta/14}}{\delta}\bigg)
  \frac{v^{\beta-\alpha}}{\delta\log v}
  \bigg),
\end{align*}
and hence we obtain
\begin{align}\label{after_int_over_u}
  &
  -\sum_{p>(\log N)^A}\sum_{\ell=1}^{\infty}d_p(\ell)(\log p) e^{-p^{\ell}/v}
  \bigg(
  \frac{-1}{\ell\log p}p^{-\ell (2-\delta/7)}
  +
  \frac{1}{\ell\log p}p^{-\ell \alpha}
  \bigg)
  \nonumber\\
    =&
    \log F(2-\delta/7) -\log F(\alpha)
    \nonumber\\
    &
    +
  O\bigg(\bigg(\frac{\log N}{\delta^2}+A(\log N)^{A\delta/14}
  +\frac{(\log N)^{A\delta/14}}{\delta}\bigg)\nonumber\\
 &\qquad\qquad\times\bigg(\frac{v^{1+\delta/14}(\log N)^{\delta/14+1/2}}{N^{\pi/2}}
+
\frac{v^{\beta-\alpha}}{\delta \log v}
\bigg)
    \bigg).
\end{align}
The left-hand side of \eqref{after_int_over_u} is
\begin{align*}
  =&
  \sum_{p>(\log N)^A}\sum_{\ell=1}^{\infty}
  \frac{d_p(\ell)e^{-p^{\ell}/v}}{\ell p^{\ell (2-\delta/7)}}
  -
  \sum_{p>(\log N)^A}\sum_{\ell=1}^{\infty}
  \frac{d_p(\ell) e^{-p^{\ell}/v}}{\ell p^{\ell \alpha}}
  \\
  =&
  O\bigg(
  \sum_{p>(\log N)^A}\sum_{\ell=1}^{\infty}\frac{1}{p^{\ell(2-\delta/7)}}
  \bigg)
  \\
  &
  -\sum_{p>(\log N)^A}\frac{d_p(1) e^{-p/v}}{p^{\alpha}}
  +
  O\bigg(
  \sum_{p>(\log N)^A}\sum_{\ell=2}^{\infty} \frac{1}{p^{\ell \alpha}}
  \bigg)
  \\
  =&
  O\bigg(
  \sum_{p>(\log N)^A}\frac{1}{p^{(2-\delta/7)}}
  \bigg)
  -\sum_{p>(\log N)^A}\frac{d_p(1) e^{-p/v}}{p^{\alpha}}
  +
  O\bigg(
  \sum_{p>(\log N)^A}\frac{1}{p^{2\alpha}}
  \bigg)
  \\
  =&
  O\bigg(
  \sum_{p>(\log N)^A}\frac{1}{p^{(2-\delta/7)}}
  \bigg)
  -\sum_{p>(\log N)^A}\frac{d_p(1) e^{-p/v}}{p^{\alpha}}.
\end{align*}
Since $2-\delta/7\geq 3/2$, by partial summation we see that
\begin{align}\label{2-2delta/7}
  &
  \sum_{p>(\log N)^A}\frac{1}{p^{(2-\delta/7)}}
  \ll 
  \int_{(\log N)^A}^{\infty} t^{(-3+\delta/7)}\frac{t}{\log t} dt
  \nonumber
  \\
  &
  \ll
  \frac{1}{A\log\log N}
  \int_{(\log N)^A}^{\infty} t^{-2+\delta/7} dt
  \ll
  \frac{(\log N)^{A\delta/7}}{(A\log\log N)(\log N)^A}.
\end{align}
Therefore, the left-hand side of \eqref{after_int_over_u} is
\begin{align*}
  &
  =-\sum_{p>(\log N)^A} \frac{d_p(1) e^{-p/v}}{p^{\alpha}}+
  O\bigg(\frac{(\log N)^{A\delta/7}}{(A\log\log N)(\log N)^A}\bigg).
\end{align*}
Further, from \eqref{2-2delta/7}, we see that
\begin{align*}
  \log F(2-\delta/7)
  =&
  \sum_{p>(\log N)^A}\sum_{\ell=1}^{\infty}\frac{d_p(\ell)}{\ell p^{\ell (2-\delta/7)}}
  \ll
  \sum_{p>(\log N)^A}\sum_{\ell=1}^{\infty}\frac{1}{\ell p^{\ell(2-\delta/7) }}
  \\
  \ll&
  \frac{(\log N)^{A\delta/7}}{(A\log\log N)(\log N)^A}.
\end{align*}
Substituting these results into \eqref{after_int_over_u}, we obtain
\begin{align}\label{tsuika}
  &
  -\sum_{p>(\log N)^A} \frac{d_p(1) e^{-p/v}}{p^{\alpha}}
  +
  \log F(\alpha)
  \nonumber\\
  \ll&
  \frac{(\log N)^{A\delta/7}}{(A\log\log N)(\log N)^A}
  \nonumber\\
  &+
\bigg(\frac{\log N}{\delta^2}+A(\log N)^{A\delta/14}
  +\frac{(\log N)^{A\delta/14}}{\delta}\bigg)
 \bigg(\frac{v^{1+\delta/14}(\log N)^{\delta/14+1/2}}{N^{\pi/2}}
+
\frac{v^{\beta-\alpha}}{\delta \log v}
\bigg).
\end{align}
Lastly, for $B\geq A$, since $e^{-p/v}\ll v/p$, again by partial summation we have
\begin{align*}
  &
  \sum_{p> (\log N)^B}
  \frac{d_p(1) e^{-p/v}}{p^{\alpha}}
  \ll
  \sum_{p\geq (\log N)^B}
  \frac{v}{p^{1+\alpha}}
  \ll 
  v\int_{(\log N)^B}^{\infty} \frac{t^{-1-\alpha}}{\log t} dt
  \\
  \ll &
  \frac{v}{B\log\log N}\int_{(\log N)^B}^{\infty} t^{-1-\alpha} dt
  \ll
  \frac{v}{(\alpha B\log\log N)(\log N)^{\alpha B}}.
\end{align*}
Inserting this estimate into \eqref{tsuika},
we arrive at the assertion of Lemma \ref{aux}.
\end{proof}
\begin{rem}\label{rem:2_ika}
If we want to extend the above lemma to the case $\alpha\geq 2$, then the inequality
\eqref{futousiki} is no longer valid, and hence, there appear more terms coming from
other poles on the right-hand side of \eqref{shifting}.   Therefore, in the case
$\alpha\geq 2$, the form of the formula stated in Lemma \ref{aux} is to be modified.
\end{rem}
\section{Proof of Theorem \ref{thm-1} and Theorem \ref{thm-2}}\label{sec-6}
%
\par
We first give one more auxiliary formula.
Define
\[
C(\alpha):=
\sum_p\sum_{\ell=2}^{\infty}\frac{d_p(\ell)}{\ell p^{\ell\alpha}}.
\]
Note that
\[
C(\alpha)\ll
\sum_p\sum_{\ell=2}^{\infty}\frac{1}{p^{\ell\alpha}}
\ll
\sum_p \frac{1}{p^{2\alpha}} \ll 1.
\]
For $7/8<1-\delta/28<\alpha < 2-\delta/7$, we have
\begin{align}\label{sum_over_p}
  &
  \log L_{p\leq (\log N)^{14/\delta}} (\alpha)
  \nonumber\\
  =&
  \sum_{p\leq (\log N)^{14/\delta}}\sum_{\ell=1}^{\infty}\frac{d_p(\ell)}{\ell p^{\ell\alpha}}
  \nonumber\\
  =&
  \sum_{p\leq (\log N)^{14/\delta}}\frac{d_p(1)}{p^{\alpha}}
  +
  \sum_{p\leq (\log N)^{14/\delta}}\sum_{\ell=2}^{\infty}\frac{d_p(\ell)}{\ell p^{\ell\alpha}}
  \nonumber\\
  =&
  \sum_{p\leq (\log N)^{14/\delta}}\frac{d_p(1)}{p^{\alpha}}
  +
  C(\alpha)
  -
  \sum_{p> (\log N)^{14/\delta}}\sum_{\ell=2}^{\infty}\frac{d_p(\ell)}{\ell p^{\ell\alpha}}
  \nonumber\\
  =&
  \sum_{p\leq (\log N)^{14/\delta}}\frac{d_p(1)}{p^{\alpha}}
  +
  C(\alpha)
  +
  O(
  \sum_{p> (\log N)^{14/\delta}}\sum_{\ell=2}^{\infty}\frac{1}{p^{\ell\alpha}}
  )
  \nonumber\\
  =&
  \sum_{p\leq (\log N)^{14/\delta}}\frac{d_p(1)}{p^{\alpha}}
  +
  C(\alpha)
  +
  O(
  \sum_{p> (\log N)^{14/\delta}}\frac{1}{p^{2\alpha}}
  )
  \nonumber\\
  =&
  \sum_{p\leq (\log N)^{14/\delta}}\frac{d_p(1)}{p^{\alpha}}
  +
  C(\alpha)
  +
  O\bigg(
  \frac{1}{(\log\log N)(\log N)^{14(2\alpha-1)/\delta}}
  \bigg).
\end{align}

Now we start the proof of Theorems \ref{thm-1} and \ref{thm-2}.
Hereafter in this section we restrict ourselves to the case
$1-\delta/28<\alpha\leq 1$, because 
from the arithmetic viewpoint, the case $\alpha\leq 1$ is more interesting.
\begin{proof}[Proof of Theorem \ref{thm-1}]
  We use Lemma \ref{aux} with $B=A$, and we put $M=\max\{1, A\delta/14\}$.
  We have
  \begin{align*}
  \log F(\alpha)
  \ll&
  \frac{v}{(A\log\log N)(\log N)^{\alpha A}}
  +
  \frac{(\log N)^{A\delta/7}}{(A\log\log N)(\log N)^A}
  \nonumber\\
  &+
\bigg(\frac{1}{\delta^2}+A\bigg)
(\log N)^M
 \bigg(\frac{v^{1+\delta/14}(\log N)^{\delta/14+1/2}}{N^{\pi/2}}
+
\frac{v^{\beta-\alpha}}{\delta \log v}
\bigg).
\end{align*}
Setting $v=(\log N)^{(M+\alpha A)/(1+\alpha-\beta)}$
%
(where $\beta=1-\delta/14$)
we have
\begin{align*}
  &
  \log F(\alpha)=
  \log L(\alpha)
  -
  \log L_{p\leq (\log N)^A}(\alpha)
  \nonumber\\
  \ll &
    \frac{(\log N)^{(M-\alpha A(\alpha-\beta))/(1+\alpha-\beta)}}{(A\log\log N)}
    +
    \frac{(\log N)^{A\delta/7}}{(A\log\log N)(\log N)^A}
    \nonumber\\
    &
    +\bigg(\frac{1}{\delta^2}+A\bigg)
    \bigg(
   \frac{(\log N)^{(1+\delta/14)(M+\alpha A)/(1+\alpha-\beta)}(\log N)^{M+\delta/14+1/2}}{\delta^2N^{\pi/2}}\bigg.\nonumber\\
&
\bigg.\qquad\qquad+
\frac{(\log N)^{(M-\alpha A(\alpha-\beta))/(1+\alpha-\beta)}}
{\delta (M+\alpha A)\log\log N }
\bigg).
\end{align*}
We now require that the error term of this estimation is $o(1)$.
This implies that $M-\alpha A(\alpha-\beta)\leq 0$, so
\[
\alpha \geq \frac{A\beta +\sqrt{A^2\beta^2 + 4AM}}{2A}
\quad{\text or}\quad
\alpha \leq \frac{A\beta -\sqrt{A^2\beta^2 + 4AM}}{2A}.
\]
Since we consider the situation $1\geq \alpha >0$, the only possible choice is
\begin{equation}\label{alpha}
1 \geq \alpha \geq \frac{A\beta +\sqrt{A^2\beta^2 + 4AM}}{2A}.
\end{equation}
The inequality
\[
1 \geq \frac{A\beta +\sqrt{A^2\beta^2 + 4AM}}{2A}
\]
gives
\[
M=\max\{1, A\delta/14\} \leq A(1-\beta)=A\delta /14.
\]
Therefore $1\leq A\delta/14$ (which is the condition required in the statement of
Theorem \ref{thm-1})
and $M=A\delta/14$.
The right-hand side of \eqref{alpha} with this value of $M$ is equal to 1.
Therefore from \eqref{alpha} we have $\alpha=1$.
This argument clarifies the reason why in Theorem \ref{thm-1} (and in Theorem
\ref{thm-2}) we only consider the value at $s=1$.

Now, for $\alpha=1$ we obtain
\begin{align}
\label{tuika_de_labelling2}
  &
  \log L(1)
  -
  \log L_{p\leq (\log N)^A}(1)
  \nonumber\\
 & \ll 
    \frac{1}{A\log\log N}
    +\biggl(\frac{1}{\delta^2}+A\biggr)
    \bigg(
   \frac{(\log N)^{A+A\delta/7+\delta/14+1/2}}{\delta^2 N^{\pi/2}}
+
\frac{1}{\delta A (1+\delta/14)\log \log N}
\bigg)\notag\\
%
%
%
%
%
%
&  \ll_{\delta,A}
    \frac{1}{\log\log N},
\end{align}
which implies the assertion of Theorem \ref{thm-1}.
\end{proof}
\par
\begin{proof}[Proof of Theorem~\ref{thm-2}]
From \eqref{sum_over_p} and \eqref{tuika_de_labelling2} we immediately obtain
\begin{equation}\label{L_1}
  \log L(1)
  -
  \sum_{p\leq (\log N)^{14/\delta}}\frac{d_p(1)}{p}
  -
  C(1)
  %
  \ll_{\delta}
    \frac{1}{\log\log N},
  \end{equation}
which is Theorem \ref{thm-2}.
\end{proof}
%
\par
In the case of Dirichlet $L$-functions,
we denote the corresponding $C(\alpha)$ by $C(\alpha,\chi)$.
\begin{rem}
  Let $z\geq 2$.    Since
  \begin{align*}
    C(1,\chi)
    =&
    \sum_{p\leq z}\sum_{\ell\geq 2}\frac{\chi(p^{\ell})}{\ell p^{\ell}}
    +
    \sum_{p> z}\sum_{\ell\geq 2}\frac{\chi(p^{\ell})}{\ell p^{\ell}}
    =
    \sum_{p\leq z}\sum_{\ell\geq 2}\frac{\chi(p^{\ell})}{\ell p^{\ell}}
    +
    O\bigg(
    \sum_{p> z}\sum_{\ell\geq 2}\frac{1}{p^{\ell}}\bigg)
    \\
    =&
    \sum_{p\leq z}\sum_{\ell\geq 2}\frac{\chi(p^{\ell})}{\ell p^{\ell}}
    +
    O\bigg(
    \sum_{p> z}\frac{1}{p(p-1)}\bigg)
    =
    \sum_{p\leq z}\sum_{\ell\geq 2}\frac{\chi(p^{\ell})}{\ell p^{\ell}}
    +
    O\bigg(\frac{1}{z}\bigg),
  \end{align*}
  we see that  $C(1,\chi)$ is twisted version of the constant $c_0$
  which appears in the formula
  \[
  \sum_{p\leq z}\sum_{\ell\geq 2}\frac{1}{\ell p^{\ell}}
  =c_0 +O\bigg(\frac{1}{z}\bigg).
  \]
It is known that $c_0+c_1=\gamma$,
  where $\gamma$ is the Euler constant and $c_1$ is the constant term of
\begin{align}\label{1/p}
  \sum_{p\leq z}\frac{1}{p}=\log\log z + c_1 +O\bigg(\frac{1}{\log z}\bigg)
\end{align}
(for $z\geq 2$; see \cite[Section 5.2]{CM}).
\end{rem}
\section{Proof of Theorem \ref{thm:main_p_s=1}}\label{sec-7}
\par
In this section we prove Theorem \ref{thm:main_p_s=1}.
Let $(\log N)^{14/\delta}\leq y\leq N^{2\pi/5}$.
In \eqref{L_alpha}, we choose $B=(\log y)/(\log\log N)$, that is, $y=(\log N)^B$.
Then $y\geq (\log N)^{14/\delta}$ implies that $B\geq 14/\delta$.
Therefore the choice $A=14/\delta$ is possible.
%
\par
From Lemma~\ref{aux}, for $1-\delta/28 <\alpha < 2-\delta/7$ we have
  \begin{align}\label{lem2-proof}
    &
   -\sum_{y \geq p >(\log N)^{14/\delta}}\frac{d_p(1)e^{-p/v}}{p^{\alpha}}
    +\log L(\alpha)-\log L_{p\leq (\log N)^{14/\delta}}(\alpha)
    \nonumber\\
    \ll &
    \frac{v}{\alpha(\log y) y^{\alpha}}
    +\frac{(\log N)^2}{(\log\log N)(\log N)^{14/\delta}}
 +
    \frac{v^{1+\delta/14}(\log N)^{\delta/14+3/2}}{\delta^2 N^{\pi/2}}+\frac{v^{\beta-\alpha}(\log N)}{\delta^3 \log v},
  \end{align}
  where $\beta=1-\delta/14$.
In view of \eqref{sum_over_p}, the left-hand side is
  \begin{align*}
  =  &
    -\sum_{y\geq p >(\log N)^{14/\delta}}\frac{d_p(1)e^{-p/v}}{p^{\alpha}}
    +\log L(\alpha)
    -
    \sum_{p\leq (\log N)^{14/\delta}}\frac{d_p(1)}{p^{\alpha}}-C(\alpha)
    \\
    &
    +O\left(
    \frac{1}{(\log\log N)(\log N)^{14(2\alpha-1)/\delta}}\right).
  \end{align*}
The error term here is absorbed into the second term on the right-hand side of
\eqref{lem2-proof}, because
$1-\alpha<\delta/28<\delta/14$ implies
$$
2-\frac{14}{\delta}>-\frac{14}{\delta}(2\alpha-1).
$$
Moreover, if $v \geq 2y$, we see that
\begin{align*}
  &
  \sum_{(\log N)^{14/\delta}< p \leq y}\frac{d_p(1)e^{-p/v}}{p^{\alpha}}
  =
  \sum_{(\log N)^{14/\delta}< p \leq y}\sum_{\ell=0}^{\infty}
  \frac{d_p(1)}{\ell ! p^{\alpha}}\bigg(\frac{-p}{v}\bigg)^{\ell}
  \\
  =&
  \sum_{(\log N)^{14/\delta}< p \leq y}
  \frac{d_p(1)}{p^{\alpha}}
  -
  \sum_{\substack{(\log N)^{14/\delta}< p \\p \leq y}}
  \frac{d_p(1)}{p^{\alpha-1}v}
  +
  \sum_{\substack{(\log N)^{14/\delta}< p \\ p \leq y}}\sum_{\ell=2}^{\infty}
  \frac{d_p(1)}{\ell ! p^{\alpha}}\bigg(\frac{-p}{v}\bigg)^{\ell}
  \\
  =&
  \sum_{(\log N)^{14/\delta}< p \leq y}
  \frac{d_p(1)}{p^{\alpha}}
  +
  O\bigg(
  \frac{1}{v}\sum_{\substack{(\log N)^{14/\delta}< p \\ p \leq y}}
  \frac{1}{p^{\alpha-1}}
  +
  \sum_{\substack{(\log N)^{14/\delta} < p \\ p \leq y}}\sum_{\ell=2}^{\infty}
  \frac{1}{\ell ! p^{\alpha}}\bigg(\frac{p}{v}\bigg)^{\ell}
  \bigg)
  \\
  =&
  \sum_{(\log N)^{14/\delta}< p \leq y}
  \frac{d_p(1)}{p^{\alpha}}
  +
  O\bigg(
  \frac{1}{v}\sum_{(\log N)^{14/\delta}< p \leq y}
  \frac{1}{p^{\alpha-1}}
  +
  \sum_{(\log N)^{14/\delta} < p \leq y}
  \frac{1}{p^{\alpha-2}v^2}
  \bigg)
  \nonumber\\
  =:&
  \sum_{(\log N)^{14/\delta}< p \leq y}
  \frac{d_p(1)}{p^{\alpha}}
  +
  E_{\alpha, y},
\end{align*} 
say.
Therefore from \eqref{lem2-proof} we obtain 
\begin{lemm}\label{lem:asymptotic_formula}
  Let $N\geq 16$.
  Under Assumptions~\ref{conti-analy}, \ref{zero-free} and \ref{convex},
  for $7/2 > \Delta > 0$, $(\log N)^{14/\delta}\leq y\leq N^{2\pi/5}$, $v \geq 2y$, 
  and for $1-\delta/28 <\alpha < 2-\delta/7$,
  we have
  \begin{align*}
    &
    \log L(\alpha) - \sum_{p< y}\frac{d_p(1)}{p^{\alpha}}
    - C(\alpha) - E_{\alpha, y}
    \nonumber\\
    \ll &
    \frac{v}{y^{\alpha}(\log y)}
    +
    \frac{(\log N)^2}{(\log\log N)(\log N)^{14/\delta}}
        +
    \frac{v^{1+\delta/14}(\log N)^{\delta/14+3/2}}{\delta^2N^{\pi/2}}+\frac{v^{1-\delta/14-\alpha}(\log N)}{\delta^3 \log v}.
  \end{align*}
  \end{lemm}
%
\par
On the term $E_{\alpha,y}$, for $1-\delta/28 <\alpha < 2-\delta/7$,  we have
\[
E_{\alpha, y}
\ll
 \frac{1}{v}\sum_{(\log N)^{14/\delta}< p < y}
  p^{1-\alpha}
  +
  \sum_{(\log N)^{14/\delta} < p < y}
  \frac{p^{2-\alpha}}{v^2}
  =\frac{1}{v}\Sigma_1+\Sigma_2,
\]
say.
It is easy to see that $\Sigma_2\ll y^{3-\alpha}/v^2 \log y$.
As for $\Sigma_1$, partial summation gives
$$
\Sigma_1\ll \frac{y^{2-\alpha}}{\log y}
+\int_{(\log N)^{14/\delta}}^y \frac{x^{1-\alpha}}{\log x}dx,
$$
and the above integral is
\begin{align}\label{chottokuwasiku}
&=\int_{(\log N)^{14/\delta}}^{\sqrt{y}}\frac{x^{1-\alpha}}{\log x}dx+\int_{\sqrt{y}}^y
\frac{x^{1-\alpha}}{\log x}dx
\ll \int_{(\log N)^{14/\delta}}^{\sqrt{y}}x^{1-\alpha}dx+\frac{1}{\log y}\int_{\sqrt{y}}^y
x^{1-\alpha}dx\notag\\
&\ll \frac{1}{2-\alpha}\left(y^{1-\alpha/2}+\frac{y^{2-\alpha}}{\log y}\right)
= \frac{y^{2-\alpha}}{(2-\alpha)\log y}\left(y^{\alpha/2-1}\log y+1\right).
\end{align}
We observe
\begin{align*}
&y^{\alpha/2-1}\log y
=\log y/e^{(1-\alpha/2)\log y}\\
&=\log y/(1+(1-\alpha/2)\log y+(1/2)(1-\alpha/2)^2(\log y)^2+\cdots)\\
&\leq 1/(1-\alpha/2),
\end{align*}
so \eqref{chottokuwasiku} is (and hence $\Sigma_1$ is)
\begin{align*}
\ll \frac{y^{2-\alpha}}{(2-\alpha)\log y}\cdot\frac{4-\alpha}{2-\alpha}
\ll \frac{y^{2-\alpha}}{(2-\alpha)^2\log y}.
\end{align*}  
We now obtain
$$
E_{\alpha,y}\ll \frac{y^{2-\alpha}}{v (2-\alpha)^2\log y}
+\frac{y^{3-\alpha}}{v^2 \log y}.
$$

Therefore Lemma~\ref{lem:asymptotic_formula} implies
\begin{align*}
  \log L(\alpha)
  =&
  \sum_{p \leq  y} \frac{d_p(1)}{p^{\alpha}} +C(\alpha)
  \\
  + &
  O\bigg(
  \frac{y^{2-\alpha}}{v(2-\alpha)^2\log y}+\frac{y^{3-\alpha}}{v^2\log y}
  +
  \frac{v}{y^{\alpha}\log y} + \frac{(\log N)^2}{(\log\log N)(\log N)^{14/\delta}}
  \\
  &
  +
  \frac{v^{1+\delta/14}(\log N)^{\delta/14+3/2}}{\delta^2N^{\pi/2}}
  +
  \frac{v^{1-\delta/14-\alpha}(\log N)}{\delta^3 \log v}
  \bigg).
\end{align*}
We put $v=y$ to obtain Theorem \ref{thm:main_p_s=1}.
\section{Proof of Theorem \ref{cor:apply_PNT} and Theorem \ref{cor:apply_PNT2}}
\label{sec-8}
Hereafter we discuss the case of Dirichlet $L$-functions.
Let $Q\geq 16$, and let $L(s,\chi)$ the Dirichlet $L$-function attached to the
primitive Dirichlet character $\chi$ mod $q (\leq Q)$.
In \eqref{cor-2-1} 
we put $y=Q$.    Then for a real number $a$ with $\delta\log Q>28a\geq 0$,
\begin{align}\label{formula:main_p_s=1_y=Q}
  &\log L(1-a/\log Q,\chi)
  =
  \sum_{p <  Q} \frac{\chi(p)}{p^{1-a/\log Q}} +C(1-a/\log Q,\chi)
  \nonumber\\
  &+ 
  O\bigg(
  \frac{e^a}{\log Q}
  +
  \frac{(\log Q)^2}{(\log\log Q)(\log Q)^{14/\delta}}
  +
  \frac{(\log Q)^{\delta/14+3/2} Q^{1+\delta/14-\pi/2}}{\delta^2}
  +
  \frac{e^a}{\delta^3 Q^{\delta/14}}
  \bigg)
\end{align}
when $\chi\notin E(\delta)$ (and hence $L(s,\chi)$ satisfies Assumption 2 with $N=Q$).
In order to deduce Theorem \ref{cor:apply_PNT}, 
we will evaluate the part
$\displaystyle{\sum_{(\log Q)^{14/\delta} < p \leq Q}}\frac{\chi(p)}{p^{1-a/\log Q}}$
of the first sum on the right-hand side of \eqref{formula:main_p_s=1_y=Q}.


\begin{proof}[Proof of Theorem \ref{cor:apply_PNT}]
Put $T=(\log Q)^3$ and $T\leq x$.
From \eqref{DL2}, \eqref{DL3}, \eqref{DL4} and Assumption~\ref{zero-free}, 
we obtain that
\begin{align}\label{1}
  &
  \sum_{p\leq x}\chi(p)\log p
  =
  \psi(x,\chi)+ O(\sqrt{x}\log x)
  \nonumber\\
  \ll&
  \frac{x^{1-\delta/7}}{\beta_1}+\sum_{|\gamma|\leq T}\frac{x^{1-\delta/7}}{|\rho|}+
  \frac{x(\log qx)^2}{(\log Q)^3}+\sqrt{x}\log x
  \nonumber\\
  \ll&
  x^{1-\delta/7}
  +\sum_{|\gamma|\leq 2}\frac{x^{1-\delta/7}}{|\rho|}
  +\sum_{2< |\gamma|\leq T}\frac{x^{1-\delta/7}}{|\rho|}+
  \frac{x(\log Qx)^2}{(\log Q)^3}+\sqrt{x}\log x.
\end{align}
From \eqref{DL1}, the first sum is
\[
\ll \frac{x^{1-\delta/7}}{\delta/7}\sum_{|\gamma|\leq 2}1
\ll \frac{x^{1-\delta/7}}{\delta}\log q
\]
(where we used the fact that, by the functional equation for Dirichlet $L$-functions
and Assumption \ref{zero-free},
there is no non-trivial zero in the region $0\leq\beta\leq\delta/7$, 
$|\gamma|\leq T$).
Also using \eqref{DL1} we see that
\begin{align*}
  \sum_{2< |\gamma|\leq T}\frac{1}{|\rho|}
  \leq &
  \sum_{2< \gamma \leq T}\frac{1}{\gamma}
  =
  \int_2^T t^{-1} dN(t,\chi)
  \nonumber\\
  =&
  \frac{N(T,\chi)}{T}-\frac{N(2,\chi)}{2}
  + \int_2^T t^{-2} N(t,\chi) dt
  \nonumber\\
  \ll&
  \log T +\log q
  + \int_2^T t^{-2}\cdot t(\log qt) dt
  \nonumber\\
  \ll &
  \log T +\log q + \int_2^T t^{-1}(\log q)dt+ \int_2^T t^{-1}(\log t)dt
  \nonumber\\
  \ll&
  \log T +\log q + (\log q)(\log T) + (\log T)^2
  \nonumber
  \\
  \ll&
  (\log Q)(\log\log Q).
\end{align*}
Therefore from \eqref{1} we obtain
\begin{equation}\label{3}
\sum_{p\leq x}\chi(p)\log p
\ll
x^{1-\delta/7}(\log Q)\bigg(\frac{1}{\delta}+\log\log Q\bigg)
+
\frac{x(\log Qx)^2}{(\log Q)^3}+\sqrt{x}\log x
\end{equation}
for $x\geq 2(\log Q)^3$.
Since $14/\delta>3$, we may apply \eqref{3} for any $x\geq (\log Q)^{14/\delta}$.
Therefore by partial summation we have
\begin{align}\label{one-more}
  &
  \sum_{(\log Q)^{14/\delta}<p\leq Q}\frac{\chi(p)}{p^{1-a/\log Q}}
  =
  \sum_{(\log Q)^{14/\delta}<p\leq Q}\frac{\chi(p)\log p}{p^{1-a/\log Q}\log p}\notag
  \\
  =&
  \frac{1}{Q^{1-a/\log Q}\log Q}\sum_{(\log Q)^{14/\delta}<p\leq Q} \chi(p)\log p\notag
  \\
  &-
  \int_{(\log Q)^{14/\delta}}^Q
  \bigg(\frac{1}{t^{1-a/\log Q}\log t}\bigg)'
  \sum_{(\log Q)^{14/\delta}<p\leq t} \chi(p)\log p \;dt\notag
  \\
  \ll&
  \frac{1}{Q\log Q}
  \bigg(
  Q^{1-\delta/7}(\log Q)\bigg(\frac{1}{\delta}+\log\log Q\bigg)
  +
  \frac{Q(\log Q)^2}{(\log Q)^3}+\sqrt{Q}\log Q
  \bigg)\notag
  \\
  &
  +
  \int_{(\log Q)^{14/\delta}}^Q
  \bigg(\frac{(1-a/\log Q)\log t+ 1}{t^{2-a/\log Q}(\log t)^2}\bigg)\notag
  \\
  &\qquad\times
  \bigg(
  t^{1-\delta/7}(\log Q)\bigg(\frac{1}{\delta}+\log\log Q\bigg)
  +
  \frac{t(\log Qt)^2}{(\log Q)^3}+\sqrt{t}\log t
  \bigg) \; dt.
\end{align}
Since $7/8<1-a/\log Q <1$, we see that
\begin{align*}
\frac{(1-a/\log Q)\log t+1}{t^{2-a/\log Q}(\log t)^2}
\ll \frac{t^{a/\log Q}}{t^2 \log t}
\leq \frac{Q^{a/\log Q}}{t^2 \log t}
=\frac{e^a}{t^2 \log t}
\end{align*}
in the above integrand, and hence
the right-hand side of \eqref{one-more} is
\begin{align*}
  \ll&
  \frac{1}{Q^{\delta/7}}\bigg(\frac{1}{\delta}+\log\log Q\bigg)
  +
  \frac{1}{(\log Q)^2}
  \\
  &
  +
  \int_{(\log Q)^{14/\delta}}^Q
  \bigg(
  \frac{e^a\log Q}{t^{1+\delta/7}\log t}\bigg(\frac{1}{\delta}+\log\log Q\bigg)
  +
  \frac{e^a(\log Qt)^2}{t(\log t)(\log Q)^3}+\frac{e^a}{t^{3/2}}
  \bigg) \; dt
  \\
  \ll&
  \frac{1}{Q^{\delta/7}}\bigg(\frac{1}{\delta}+\log\log Q\bigg)
  +
  \frac{1}{(\log Q)^2}
  \\
  &
  +
  \frac{\delta \log Q}{\log \log Q}
  \bigg(\frac{1}{\delta}+\log\log Q\bigg)
  \int_{(\log Q)^{14/\delta}}^Q
  \frac{e^a}{t^{1+\delta/7}}
  \; dt
  \\
  &
  +
  \int_{(\log Q)^{14/\delta}}^Q
  \frac{e^a(\log Q)^2}{t(\log t)(\log Q)^3}
  \; dt
  +\frac{e^a}{(\log Q)^{7/\delta}}
  \\
  \ll&
  \frac{1}{Q^{\delta/7}}\bigg(\frac{1}{\delta}+\log\log Q\bigg)
  +
  \frac{1}{(\log Q)^2}
  +
    \frac{e^a}{(\log Q)(\log\log Q)}\bigg(\frac{1}{\delta}+\log\log Q\bigg)
  \\
  &
  +
  \int_{(\log Q)^{14/\delta}}^Q
  \frac{e^a}{t(\log t)(\log Q)}
  \; dt
  +
  \frac{e^a}{(\log Q)^{7/\delta}}.
\end{align*}
Since the last integral is $\ll e^a \log\log Q/\log Q$,
We now obtain
\begin{align*}
\sum_{(\log Q)^{14/\delta}<p\leq Q}\frac{\chi(p)}{p^{1-a/\log Q}}
  \ll&
  \frac{e^a}{(\log Q)(\log\log Q)}\bigg(\frac{1}{\delta}+\log\log Q\bigg)
  +
  \frac{e^a\log\log Q}{\log Q}.
\end{align*}
Substituting this estimate into \eqref{formula:main_p_s=1_y=Q},
we obtain
  \begin{align}\label{eeee}
    &
    \log L(1-a/\log Q,\chi)
    -
    \sum_{p\leq (\log Q)^{14/\Delta}}\frac{\chi(p)}{p^{1-a/\log Q}}
    -
    C(1-a/\log Q,\chi)
    \nonumber\\
    \ll &
  \frac{e^a}{(\log Q)(\log\log Q)}\bigg(\frac{1}{\delta}+\log\log Q\bigg)
  +
  \frac{e^a \log\log Q}{\log Q}
    \nonumber\\
    &
    +
    \frac{(\log Q)^2}{(\log\log Q)(\log Q)^{14/\delta}}
    +
    \frac{(\log Q)^{\delta/14+3/2} Q^{1+\delta/14-\pi/2}}{\Delta^2}
    +
    \frac{e^a}{\Delta^3 Q^{\Delta/14}}
    \end{align}
when $\chi\notin E(\delta)$.
Since $\delta<7/2$, the right-hand side of the above is
    \begin{align}\label{eeeee}
    \ll_{\Delta} 
    \frac{e^a \log\log Q}{\log Q}.
  \end{align}
  This implies the first formula \eqref{formula_th4_1} of Theorem \ref{cor:apply_PNT}.

  Lastly, using \eqref{sum_over_p} for Dirichlet $L$-functions
  with $N=Q$ 
  and $\alpha=1-a/\log Q$,
we find
\begin{align*}
  &
  \sum_{p\leq (\log Q)^{14/\Delta}}\frac{\chi(p)}{p^{1-a/\log Q}} +C(1-a/\log Q,\chi)
  \\
=&
\log L_{p\leq (\log Q)^{14/\delta}}(1-a/\log Q,\chi)+O\left(
\frac{1}{(\log Q)^3 \log\log Q}\right),
\end{align*}
since $\delta\log Q> 28a$ and $14/\delta > 4$.
Substituting this into the right-hand side of \eqref{formula_th4_1}, and taking
the exponentials of the both sides, we obtain \eqref{formula_th4_2}.
\end{proof}

\begin{rem}\label{4implies1}
  In the case $0< \Delta <1$, using \eqref{1/p}
  we see that, for $(\log Q)^{\Delta} < y \leq (\log Q)$,
  \begin{align*}
    &
    \left|\sum_{y < p \leq (\log Q)^{14/\delta} }\frac{\chi(p)}{p}\right|
    \leq
    \sum_{(\log Q)^{\delta} < p \leq (\log Q)^{14/\delta} }\frac{1}{p}
    \\
    \leq &
    \log\log (\log Q)^{14/\delta}
    -
    \log\log (\log Q)^{\Delta} +O\bigg(\frac{1}{\log (\log Q)^{\delta}}\bigg)
    \\
    =&
    \log \frac{14}{\Delta}
    +
    \log\log\log Q
    -
    \log \Delta
    -
    \log\log\log Q
    +
    O\bigg(\frac{1}{\delta\log \log Q}\bigg)
    \\
    = &
    \log \frac{14}{\Delta^2}
    +
    O\bigg(\frac{1}{\delta\log \log Q}\bigg)
    \ll_{\Delta}
    1.
\end{align*}
Therefore \eqref{formula_th4_1} essentially implies Proposition~\ref{lem:MW}. 
\end{rem}

\begin{proof}[Proof of Theorem \ref{cor:apply_PNT2}]
We evaluate the sum
\[
\sum_{(\log Q)^A<p\leq (\log Q)^{14/\delta}}\frac{\chi(p)}{p^{1-a/\log Q}},
\]
where $0<A<14/\delta$.
Using Gallagher's large sieve inequality
\[
\sum_{q\leq Q}{\sum_{\chi \text{mod}\;q}}^*
\left|\sum_{M< m \leq M+M'} a_n\chi(n)\right|^2
\leq
(Q^2+\pi M') \sum_{M< m\leq M+M'}|a_n|^2,
\]
where ${\sum}^*$ means that the summation is restricted to primitive 
characters
(see \cite[Theorem 2.5]{Mont}),
we see that
\begin{align}\label{large-sieve}
  &\sum_{q\leq Q}{\sum_{\chi \text{mod}\;q}}^*
  \left|
  \sum_{(\log Q)^A< p \leq (\log Q)^{14/\delta}} \frac{\chi(p)}{p^{1-a/\log Q}}
  \right|^2\nonumber\\
  &\quad\ll
  (Q^2+(\log Q)^{14/\delta}-(\log Q)^A)
  \sum_{(\log Q)^A< p \leq (\log Q)^{14/\delta}}\frac{(\log Q)^{28a/(\delta\log Q)}}{p^2}
  \nonumber\\
  &\quad\ll
  (Q^2+(\log Q)^{14/\delta}) \frac{(\log Q)^{28a/(\delta\log Q)}}{ (\log Q)^A}.
\end{align}
Let $\mathcal{X}$ be the set of all primitive Dirichlet characters
mod $q\leq Q$ which satisfy
\[
\left|\sum_{(\log Q)^A< p \leq (\log Q)^{14/\delta}} \frac{\chi(p)}{p^{1-a/\log Q}}\right|
>
\frac{f(Q)}{(\log Q)^{A/2}},
\]
where $f(Q)$ is as in the statement of Theorem \ref{cor:apply_PNT2}.
Then obviously
\[
\sum_{q\leq Q}{\sum_{\chi \text{mod}\; q}}^*
\left|\sum_{(\log Q)^A< p \leq (\log Q)^{14/\delta}} \frac{\chi(p)}{p^{1-a/\log Q}}\right|^2
>
\frac{(\#\mathcal{X})f(Q)^2}{(\log Q)^{A}}.
\]
Combining this and \eqref{large-sieve} we obtain
\begin{align}\label{estimate_X}
  \#\mathcal{X}
  \ll
  (Q^2+(\log Q)^{14/\delta}) \frac{(\log Q)^{A+(28a)/(\delta\log Q)}}{f(Q)^2(\log Q)^A}
    \ll_{\delta} \frac{Q^2(\log Q)^{28a/(\delta\log Q)}}{f(Q)^2}.
\end{align}
If $\chi\notin\mathcal{X}$, then
\[
\left|\sum_{(\log Q)^A< p \leq (\log Q)^{14/\delta}} \frac{\chi(p)}{p^{1-a/\log Q}}\right|
\leq
\frac{f(Q)}{(\log Q)^{A/2}}.
\]
From this inequality and \eqref{eeeee}, 
for $\chi\notin E(\delta)\cup\mathcal{X}$ we obtain
\begin{align}\label{ffff}
&  \log L(1-a/\log Q,\chi)- 
  \sum_{p\leq (\log Q)^A}\frac{\chi(p)}{p^{1-a/\log Q}} 
  -C(1-a/\log Q,\chi)\notag\\
 & \qquad  \ll_{\delta}
    \frac{e^a \log\log Q}{\log Q}+ \frac{f(Q)}{(\log Q)^{A/2}}.
\end{align}
If $f(Q)=o((\log Q)^{A/2})$,
the right-hand side of the above formula is $o(1)$ as $Q\to\infty$.

Using \eqref{estimate_E} and \eqref{estimate_X} we see that
  \[
  \#(E(\delta)\cup\mathcal{X})
  \ll_{\Delta} Q^{\Delta} + \frac{Q^2(\log Q)^{28a/(\delta\log Q)}}{f(Q)^2},
  \]
while it is proved by Jager \cite{J} that
  \[
  \sum_{q\leq Q}{\sum_{\chi ({\rm mod}\;q)}}^{\!\!\!\!\!*} 1 =\frac{18}{\pi^4}Q^2
  +O(Q(\log Q)^2).
  \]
  Therefore if $\delta<2$ and $(\log Q)^{14a/(\delta\log Q)}=o( f(Q))$,
  then $\#(E(\delta)\cup\mathcal{X})=o(Q^2)$, and hence
we can conclude that \eqref{ffff} holds for almost all $\chi$.
\end{proof}

\par
K. Matsumoto: Graduate School of Mathematics, Nagoya University,
Chikusa-ku, Nagoya 464-8602, Japan\\
and\\
Center for General Education, Aichi Institute of Technology,
Yakusa-cho, Toyota 470-0392, Japan\\
E-mail address: kohjimat@math.nagoya-u.ac.jp
\par
Y. Umegaki: Department of Mathematical and Physical Sciences,
Nara Women's University, Kitauoya-nishimachi, Nara 630-8506, Japan\\
E-mail address: ichihara@cc.nara-wu.ac.jp
\end{document}